\documentclass[a4paper,12pt,reqno]{amsart}
\usepackage{amsfonts}
\usepackage{amsmath}
\usepackage{amssymb}
\usepackage[a4paper]{geometry}
\usepackage{mathrsfs}
\usepackage{csquotes}
\usepackage{enumitem}
\usepackage{dirtytalk}
\usepackage{comment}
\usepackage{nccmath}
\usepackage{soul}
\usepackage{breakcites}

\usepackage{tikz}
\usetikzlibrary{calc, arrows.meta, angles, quotes}

\usepackage[colorlinks]{hyperref}
\renewcommand\eqref[1]{(\ref{#1})} 
\numberwithin{equation}{section}
\theoremstyle{plain}
\newtheorem{thm}{Theorem}[section]

\newtheorem{cor}[thm]{Corollary}
\newtheorem{lem}[thm]{Lemma}

\theoremstyle{definition}
\newtheorem{defn}[thm]{Definition}
\newtheorem{rem}[thm]{Remark}

\newcommand{\N}{\mathbb{N}}
\newcommand{\C}{\mathbb{C}}
\newcommand{\R}{\mathbb{R}}
\newcommand{\re}{\operatorname{Re}}

\begin{document}

   \title[General weighted discrete $p$-Hardy inequalities]
 {Sharp forms and quantitative stability for general weighted discrete $p$-Hardy inequalities}

\author[N. Yessirkegenov]{Nurgissa Yessirkegenov}
\address{
  Nurgissa Yessirkegenov:
  \endgraf
  KIMEP University, Almaty, Kazakhstan
     \endgraf
  {\it E-mail address} {\rm nurgissa.yessirkegenov@gmail.com}
  }

  \author[A. Zhangirbayev]{Amir Zhangirbayev}
\address{
  Amir Zhangirbayev:
 \endgraf
   SDU University, Kaskelen, Kazakhstan
  \endgraf
  and 
  \endgraf
  Institute of Mathematics and Mathematical Modeling, Kazakhstan
   \endgraf
  {\it E-mail address} {\rm
amir.zhangirbayev@gmail.com}
  }

\thanks{This research is funded by the Committee of Science of the Ministry of Science and Higher Education of the Republic of Kazakhstan (Grant No. AP23490970).}


     \begin{abstract}
     In this paper, we provide a sharp remainder term for the general weighted discrete $p$-Hardy inequality. By choosing appropriate weights and specifying $1<p<\infty$, we are able to recover the identity by Krej{\v{c}}i{\v{r}}{\'\i}k-\v{S}tampach \cite[Theorem 1]{krejcirik2022sharp}, obtain the sharp form of the $p$-Hardy inequality by Fischer-Keller-Pogorzelski \cite[Theorem 1]{fischer2023improved} and generalize the power weighted inequality by Gupta \cite[Theorem 2.1]{gupta2022discrete} with a sharp remainder. In addition, we prove a quantitative stability type result, thereby showing that the deficit of the discrete $p$-Hardy inequality controls the weighted distance to the family of non-trivial minimizers. 
     \end{abstract}
     \maketitle

\section{Introduction}\label{sec intro}

More than a hundred years ago, G. H. Hardy received a letter \cite{landau1921letter} from E. Landau dated 21 June 1921, containing a proof of the following result: let $1<p<\infty$, then for all non-negative real-valued sequences $\{a_n\}_{n=1}^{\infty}$, we have
\begin{align}\label{p-discrete hardy}
\sum_{n=1}^{\infty} a_n^p \geq \left(\frac{p-1}{p}\right)^p \sum_{n=1}^{\infty} \left(\frac{a_1 + a_2 + \ldots + a_n}{n}\right)^p,
\end{align}
where the constant $\left(\frac{p-1}{p}\right)^p$ is sharp and equality holds if and only if $a_{n}=0$ for all $n$. Nowadays a more modern and equivalent form of (\ref{p-discrete hardy}) is used in terms of compactly supported sequences: let $1<p<\infty$ and $\mathbb{N}_0:=\{0, 1, 2, 3, \ldots\}$, then for all $u\in C_{c}(\mathbb{N}_0)$, we have
\begin{align}\label{p-hardy dis 2}
\sum_{n=1}^{\infty} |u(n) - u(n-1)|^p \geq \left(\frac{p-1}{p}\right)^p \sum_{n=1}^{\infty} \frac{|u(n)|^p}{|n|^p}
\end{align}
with $u(0)=0$. The inequality (\ref{p-discrete hardy}) has a very rich history behind it and we refer to \cite{kufner2006prehistory, kufner2007hardy, persson2024hardy} for excellent surveys on the subject. Now, for context, it is also important to mention a continuous version of (\ref{p-discrete hardy}) that was developed shortly after (\ref{p-discrete hardy}):
\begin{align}\label{p-hardy continuous}
  \int_{0}^{\infty} |\varphi'(x)|^{p} \, \mathrm{d}x
  \geq \left(\frac{p-1}{p}\right)^{p}
  \int_{0}^{\infty} \frac{|\varphi(x)|^{p}}{x^{p}} \, \mathrm{d}x
\end{align}
for smooth compactly supported functions $\varphi \in C_{0}^{\infty}(\mathbb{R}_{+})$ in $\mathbb{R}_{+} \equiv (0, \infty)$. Interestingly enough, the inequality (\ref{p-hardy continuous}) has had a much more extensive development than (\ref{p-hardy dis 2}) (see e.g., \cite{mazya2013sobolev, balinsky2015analysis, ruzhansky2019hardy}). The reason for this is partially due to the breakdown of calculus in the discrete setting, which makes the transfer of proofs from continuous to discrete much harder. For example, in the continuous framework, due to polar coordinates, the multidimensional version of (\ref{p-hardy continuous}) can be obtained almost immediately. However, there is no real substitute for polar coordinates in the discrete setting. Although the explicit form of the optimal constant in any dimension is still open, a lot has been done already (see \cite{rozenblum2009spectral, rozenblum2014spectral, kapitanski2016continuous, keller2018optimal, gupta2023hardy, gupta2024discrete}).

Apart from the lack of suitable techniques, the discrete Hardy inequality (\ref{p-discrete hardy}) also exhibits some interesting phenomena that completely differ from its continuous counterpart. In particular, in the continuum, it is well-known that both the constant $\left(\frac{p-1}{p}\right)^{p}$ and the weight $\left(\frac{p-1}{p}\right)^{p}\frac{1}{|x|^{p}}$ are sharp and cannot be improved further. One might expect the discrete analog (\ref{p-hardy dis 2}) to have similar behavior. While the constant in (\ref{p-hardy dis 2}) is sharp, the weight sequence $\left(\frac{p-1}{p}\right)^{p}\frac{1}{|n|^p}$ is not. In fact, in \cite{keller2018improved}, Keller, Pinchover and Pogorzelski proved that for $p=2$, the classical weight sequence $w^{H}_{2}:=\frac{1}{4n^{2}}$ can be replaced by the following
\begin{align*}
w_{2}(n) = 2 - \left(\left(1 + \frac{1}{n}\right)^{1/2} + \left(1 - \frac{1}{n}\right)^{1/2}\right)>w^{H}_{2},
\end{align*}
which is strictly larger than the classical one. In \cite{keller2018optimal}, the same authors showed that the new weight $w(n)$ is indeed sharp by showing criticality, null-criticality and optimality near infinity. Later, the same was done for any $1<p<\infty$ by Fischer, Keller and Pogorzelski \cite{fischer2023improved}: let $p > 1$. Then, for all $u \in C_c(\mathbb{N}_{0})$ with $u(0) = 0$, we have
\begin{align}\label{FKP}
  \sum_{n=1}^{\infty} |u(n) - u(n-1)|^{p}
  \geq \sum_{n=1}^{\infty} w_p(n) \, |u(n)|^{p},
\end{align}
where $w_p$ is a strictly positive function given by
\begin{align*}
  w_p(n) = \left(1 - \left(1 - \frac{1}{n}\right)^{\frac{p-1}{p}}\right)^{p-1}
  - \left(\left(1 + \frac{1}{n}\right)^{\frac{p-1}{p}} - 1\right)^{p-1}.
\end{align*}
Furthermore, $w_p$ is optimal and we have for all $n \in \mathbb{N}$
\begin{align}\label{fkp ineq weight}
w_p(n) > w_p^{H}(n):=\left(\frac{p-1}{p}\right)^{p}\frac{1}{|n|^p}.
\end{align}
Clearly, the discrete setting requires very careful treatment as the intuition and techniques developed in the continuous setting do not transfer directly. Now, in the next part of this section, we would like to turn to results, which are closely related to our contribution in this paper.

After Keller, Pinchover and Pogorzelski published their results, there has been an influx of development over the next decade. One of the first results was obtained in \cite{krejcirik2022sharp}, where Krej{\v{c}}i{\v{r}}{\'\i}k and \v{S}tampach obtained a sharp form of the Keller-Pinchover-Pogorzelski inequality and optimality of the weight $w_{2}(n)$: for all $u \in C_c(\mathbb{N}_0)$ with $u(0) = 0$, we have the identity
\begin{multline*}
  \sum_{n=1}^{\infty} |u(n) - u(n-1)|^2=\sum_{n=1}^{\infty} w_2(n) |u(n)|^2
  \\+ \sum_{n=2}^{\infty} \left|
    \sqrt[4]{\frac{n-1}{n}} \, u(n)
    - \sqrt[4]{\frac{n}{n-1}} \, u(n-1)
  \right|^2.
\end{multline*}
In particular, the corresponding inequality holds and, in addition, is optimal. 

Then, Gupta, in \cite{gupta2022discrete}, derived a condition for obtaining general weighted Hardy inequalities for $p=2$: let us define a combinatorial Laplacian for a real-valued function $\varphi$ on $\mathbb{N}_{0}$ as
\begin{align*}
\Delta\varphi(n) :=
\begin{cases}
  \varphi(n) - \varphi(n-1) + \varphi(n) - \varphi(n+1) & \text{for } n \geq 1, \\
  \varphi(n) - \varphi(n+1) & \text{for } n = 0.
\end{cases}
\end{align*}
Let $v$ and $w$ be non-negative functions on $\mathbb{N}$. Assume there exists a function
$\varphi : \mathbb{N}_0 \to [0, \infty)$, which is positive on $\mathbb{N}$ such that
\begin{align}\label{gupta cond}
  \Big(\Delta\varphi(n) v(n) - (\varphi(n+1) - \varphi(n))(v(n+1) - v(n))\Big)
  \geq w(n)\varphi(n)
\end{align}
for all $n \in \mathbb{N}$. Then the following inequality holds true
\begin{align}\label{general weighted gupta}
  \sum_{n=1}^{\infty} v(n)|u(n) - u(n-1)|^2 
  \geq \sum_{n=1}^{\infty} w(n) |u(n)|^2
\end{align}
for $u \in C_c(\mathbb{N}_0)$ and $u(0) = 0$. 

Under a certain choice of weights $(v(n), \varphi(n), w(n))$, Gupta was able to obtain a discrete Hardy inequality with power weights $n^{\alpha}$, which in the special case recovers (\ref{FKP}, $p=2$). The idea behind the proof of (\ref{general weighted gupta}) actually comes from a well-known technique in the continuous setting, called the supersolution technique (see \cite{sekar2006role, adimurthi2006optimal, frank2008non, ghoussoub2011bessel, ghoussoub2013functional, cazacu2021method}). The method allows one to obtain general weighted Hardy inequalities by just finding a solution to a differential equation. As one can see, the condition (\ref{gupta cond}) essentially acts as a \say{differential equation} in the discrete framework, which if satisfied for a given pair of weights, gives the desired inequality. 

This approach was later extended to locally finite graphs with sharp remainder terms by Huang and Ye \cite{huang2024onedimensional}: 
let $v(n)$ be a non-negative function on $\mathbb{N}$ and $\varphi: \mathbb{N}_{0}\to [0,\infty)$ be such that $\varphi(n) > 0$
for all $n \geq 1$ and $\varphi(0)=0$. Then, for all complex-valued $u \in C_c(\mathbb{N}_{0})$, we have the identity
\begin{multline}\label{huang-ye}
  \sum_{n=1}^{\infty} v(n) |u(n)-u(n-1)|^2
  + \sum_{n=1}^{\infty} \frac{\operatorname{div}(v(n)(\nabla \varphi(n)))}{\varphi(n)} |u(n)|^2
  \\= \sum_{n=1}^{\infty} v(n+1) \left|
    \sqrt{\frac{\varphi(n)}{\varphi(n+1)}} \, u(n+1)
    - \sqrt{\frac{\varphi(n+1)}{\varphi(n)}} \, u(n)
  \right|^2,
\end{multline}
where $\nabla \varphi(n) := \varphi(n) - \varphi(n-1)$ and
\begin{align*}
\operatorname{div} h(n) := h(n+1) - h(n).
\end{align*}

An important next step in this direction was to consider higher-order generalizations of (\ref{huang-ye}). In the literature, these are often called Hardy-Rellich-Birman inequalities due to Birman \cite{birman1961spectrum} who was the first to generalize (\ref{p-hardy continuous}, $p=2$) to higher order derivatives. Indeed, this exact analysis was done in \cite{stampach2024optimal} by {\v{S}}tampach and Waclawek, where they not only successfully recovered optimal discrete Keller-Pinchover-Pogorzelski, Rellich, Huang-Ye and Birman weights, but also proved and improved the conjecture posed by Gerh{\'a}t-Krej{\v{c}}i{\v{r}}{\'\i}k-{\v{S}}tampach in \cite{gerhat2025improved}. We also refer to \cite{ciaurri2018hardy, gupta2024onedimensional, gupta2024discrete, gerhat2025criticality} for other similar results.

The most recent result in this line of research is again by {\v{S}}tampach and Waclawek \cite{stampach2026hardy}. There, the authors extended the argument from \cite{stampach2024optimal} to any $1<p<\infty$: let $p > 1$, $v(n) \geq 0$ for all $n \in \mathbb{N}$, $\varphi(n+1) \geq \varphi(n) > 0$
for all $n \in \mathbb{N}$, and $\varphi(0) = 0$. Then for any complex-valued $u \in C_c(\mathbb{N}_0)$
with $u(0) = 0$, we have the inequality
\begin{align}\label{lp stampach}
  \sum_{n=1}^{\infty} v(n) \, |\nabla u(n)|^{p}
  \geq - \sum_{n=1}^{\infty} \frac{\operatorname{div}(v(n)(\nabla \varphi(n))^{p-1})}{\varphi(n)^{p-1}} \, |u(n)|^{p},
\end{align}
where $\nabla \varphi(n) := \varphi(n) - \varphi(n-1)$ and
\begin{align*}
  \operatorname{div} h(n) := h(n+1) - h(n).
\end{align*}

When $p=2$, the inequality (\ref{lp stampach}) recovers the inequality case of (\ref{huang-ye}). However, it does not recover the remainder term of (\ref{huang-ye}). In this paper, we provide a framework for deriving sharp remainder terms in the discrete setting for any $1<p<\infty$. To be precise, we prove the following result: let us define a $C_p$-functional for one-dimensional complex-valued variables $\xi, \eta \in \C$ by
\begin{align*}
C_p(\xi, \eta) := |\xi|^p - |\xi - \eta|^p - p|\xi - \eta|^{p-2}\re \langle \xi-\eta, \eta \rangle \geq 0,
\end{align*}
where $1<p<\infty$. Let $v$ be a non-negative function on $\N$, $w$ be a real-valued function on $\mathbb{N}$ and let $\varphi(n+1)\geq\varphi(n)>0$ for all $n\in \mathbb{N}$ with $\varphi(0)=0$ and
\begin{align}\label{cond intro}
v(n)\bigl(\varphi(n) - \varphi(n-1)\bigr)^{p-1} - v(n+1)\bigl(\varphi(n+1) - \varphi(n)\bigr)^{p-1}\geq w(n)\varphi(n)^{p-1}.
\end{align}
Then, for every complex-valued $u \in C_c(\N_0)$ with $u(0) = 0$, we have
\begin{align}\label{eq hardy identity intro}
\sum_{n=1}^{\infty} v(n)|u(n) - u(n-1)|^p
\geq \sum_{n=1}^{\infty} w(n)|u(n)|^p + \sum_{n=2}^{\infty} v(n)R_p(u(n),\varphi(n)),
\end{align}
where at each index $n\geq 2$, we have
\begin{multline*}
R_p(u(n),\varphi(n))
:= C_p\!\left(u(n) - u(n-1),\;\frac{\varphi(n-1)}{\varphi(n)}\,u(n) - u(n-1)\right) \\ + \varphi(n-1)\bigl(\varphi(n) - \varphi(n-1)\bigr)^{p-1}\,C_p\!\left(\frac{u(n-1)}{\varphi(n-1)},\;\frac{u(n-1)}{\varphi(n-1)} - \frac{u(n)}{\varphi(n)}\right).
\end{multline*}
Consequently, if $v(n)>0$ for every $n\geq2$, then the remainder term in (\ref{eq hardy identity}) vanishes if and only if
$
u(n)=c\phi(n)
$
for every $n\in\mathbb{N}$ and some $c\in\mathbb{C}$. Moreover, we have equality in (\ref{eq hardy identity intro}) if (\ref{cond intro}) is satisfied with equality.

The existing results on $C_p$-functional give the non-negativity of the remainder term $R_{p}(u(n),\varphi(n))$. This allows us to drop it, thereby recovering the general weighted discrete $p$-Hardy inequality (\ref{lp stampach}) by {\v{S}}tampach and Waclawek. We also note that compared to (\ref{lp stampach}), we write (\ref{eq hardy identity intro}) in double weighted form for simplicity and convenience.

As we said, (\ref{lp stampach}) does not recover the remainder term of (\ref{huang-ye}) as its proof is based on the algebraic inequality from \cite[Lemma 2.6]{frank2008non}. As a result, the equality case is lost. Nevertheless, when $p=2$, in our identity (\ref{eq hardy identity intro}), we are able to get exactly the identity of Huang and Ye (\ref{huang-ye}) (see Section \ref{sec proofs 3} for more details).

In addition, similarly as in \cite{stampach2026hardy}, we can specify the weights $v(n)$ and $\varphi(n)$ to obtain previously known and new discrete Hardy inequalities and identities. For instance, choosing 
\begin{multline*}
v(n) := \begin{cases}
  0 & \text{if } n = 1, \\
  (n-1)^{\alpha} & \text{if } n \geq 2,
\end{cases}
\\
\varphi(n)=\hat{\varphi}_{Cop}(n) := \begin{cases}
  0 & \text{if } n = 0, \\
  \varphi_{Cop}(n+1) & \text{if } n \geq 1,
\end{cases}
\end{multline*}
where $\varphi_{Cop}(n) = \Gamma\!\left(n + 1 - \frac{\alpha+1}{p}\right) / \Gamma(n)$ and $\Gamma$ is the Euler Gamma function, implies the following version of the Copson inequality (see \cite{copson1928note, bennett1987elementary, das2023improvements, das2025improved, das2025improvement} and \cite[Theorem 3]{stampach2026hardy}) with remainder terms for $\alpha<0$ and $1<p<\infty$: let $u \in C_{c}(\mathbb{N}_0)$ be a complex-valued compactly supported sequence with $u(0)=u(1)=0$, then
\begin{multline*}
\sum_{n=1}^{\infty} n^{\alpha} |u(n)-u(n-1)|^{p}
\geq \left(\frac{p - \alpha - 1}{p}\right)^{p} \sum_{n=1}^{\infty} (n+1)^{\alpha-p} |u(n)|^{p}\\+\sum_{n=2}^{\infty}n^{\alpha}R_{p}\left(u(n), \hat{\varphi}_{Cop}(n)\right),
\end{multline*}
where the constant $\left(\frac{p - \alpha - 1}{p}\right)^{p}$ is optimal. 

We are also able to obtain the sharp remainder term of the inequality by Gupta \cite[Theorem 2.1]{gupta2022discrete} for any $1<p<\infty$ and complex-valued sequences by taking
\begin{align*}
v(n)=n^{\alpha} \quad \text{and} \quad \varphi(n)=\varphi_{\beta}(n)=n^{\beta} \quad \text{for} \quad \alpha \in \mathbb{R} \quad \text{and} \quad \beta> 0.
\end{align*}
This gives us the next result: let $u\in C_{c}(\mathbb{N}_0)$ be a complex-valued compactly supported sequence with $u(0)=0$, then
\begin{align}\label{gen gupta}
\sum_{n=1}^{\infty} n^{\alpha}|u(n) - u(n-1)|^p 
= \sum_{n=1}^{\infty}w_{p,\alpha,\beta}(n)|u(n)|^{p}+\sum_{n=2}^{\infty}n^{\alpha}R_{p}\left(u(n), \varphi_{\beta}(n)\right),
\end{align}
where
\begin{align*}
w_{p,\alpha,\beta}(n)=n^{\alpha} \left[\left(1 - \left(1 - \frac{1}{n}\right)^{\beta}\right)^{p-1}
- \left(1 + \frac{1}{n}\right)^{\alpha} \left(\left(1 + \frac{1}{n}\right)^{\beta} - 1\right)^{p-1}\right]
\end{align*}
for $n\geq2$, and $w_{p,\alpha,\beta}(1):=1-2^{\alpha}\left(2^{\beta}-1\right)^{p-1}$. Moreover, when $p=2$, identity \eqref{gen gupta} extends to $\beta\in\mathbb{R}$. 

After dropping the remainder term, the $p=2$ case of (\ref{gen gupta}) implies \cite[Theorem 2.1]{gupta2022discrete} for $\beta\in \mathbb{R}$. Setting $\alpha=0$ and $\beta=\frac{p-1}{p}$ yields the sharp version of the Fischer-Keller-Pogorzelski inequality (\ref{FKP}). If we combine this condition with $p=2$ (that is, $p=2$, $\alpha=0$, and $\beta=\frac{p-1}{p}$), we deduce the identity established by Krej{\v{c}}i{\v{r}}{\'\i}k and \v{S}tampach \cite[Theorem 1]{krejcirik2022sharp}. Additionally, specifying $\beta=-\frac{\alpha-p+1}{p}$, we get the sharp form of \cite[Theorem 4.11]{barki2024sharp} by Barki with the condition that $0\leq \alpha<p-1$. 

Finally, we would like to end this section with an application of (\ref{eq hardy identity intro}) to the question of stability (or rather quantitative forms of (\ref{p-hardy dis 2})). This question concerns the limiting behavior of (\ref{p-hardy dis 2}). Particularly, one asks: given some functional inequality $\mathcal{F}(u)\geq c\mathcal{G}(u)$ and a set of minimizers $\mathcal{M}$, 
\begin{align*}
\mathcal{F}(u)- c\mathcal{G}(u) \to 0 \quad \implies \quad u \to \tilde{u}\in \mathcal{M} \quad ?
\end{align*}
A quantitative version of this question seeks an estimate of the form
$$
\mathcal{F}(u)- c\mathcal{G}(u)\geq \Psi\bigl(d(u,\mathcal M)\bigr),
$$
where $d(u,\mathcal M)$ measures the distance from $u$ to the set of minimizers $\mathcal M$, and $\Psi:[0,\infty)\to[0,\infty)$ satisfies $\Psi(t)>0$ for every $t>0$. Such an estimate quantifies the principle that a small deficit forces $u$ to be close to $\mathcal M$. This kind of analysis was first done by Bianchi and Egnell for the Sobolev inequality \cite{bianchi1991note}. Answering such a question for various inequalities (e.g. Poincar\'e, Gagliardo–Nirenberg and etc.) turns out to have beautiful connections in the theory of optimal transport, symmetrization theory, spectral analysis and fast diffusion flows (see \cite{cianchi2009sharp, brasco2012sharp, figalli2022sharp, dolbeault2025sharp, bonforte2025stability, yessirkegenov2026stability}).

Now, for the Hardy inequality specifically a lot has been done in the continuous setting. First, let us note that for both continuous (\ref{p-hardy continuous}) and discrete (\ref{p-hardy dis 2}) versions, the Hardy inequality actually admits no non-trivial minimizers. Although the \say{formal} minimizers are known to be in the family $c|x|^{\frac{p-1}{p}}$ (or $cn^{\frac{p-1}{p}}$ in the discrete) for $c\in \mathbb{R}\backslash \{0\}$, they do not lie in the appropriate space of functions. This differs from other inequalities such as Sobolev or Poincar\'e, whose minimizers are admissible in the corresponding function spaces.

In \cite{cianchi2008hardy}, Cianchi and Ferone showed that despite the fact that the continuous (multidimensional) Hardy inequality (\ref{p-hardy continuous}) admits no minimizers (unless $u=0$), any minimizing sequence of functions must approach the family of non-trivial minimizers $c|x|^{\frac{p-1}{p}}$. Later, similar results have been obtained in other settings \cite{sano2017scale, ioku2017hardy, sano2018scaling, bez2018stability, sano2018improvements, ruzhansky2018note, ruzhansky2019critical, banerjee2026sharp}. However, to the best of our knowledge, no such results are available in the discrete setting. There is a related quantitative result by Barki \cite{barki2024sharp}, where stability was established by showing that the deficit of the classical discrete $p$-Hardy inequality \eqref{p-hardy dis 2} controls a weighted distance to the trivial zero sequence. While this provides a rate of vanishing in a specific weighted norm, it does not capture the actual asymptotic shape of the minimizing sequences. In this paper, we fill this gap by proving that the deficit controls the weighted distance to the family of non-trivial minimizers $cn^{\frac{p-1}{p}}$.

The paper is organized as follows. In Section~\ref{sec pre}, we collect the necessary preliminary results, including the definition, key properties of the $C_p$-functional and state the auxiliary results. Section~\ref{sec main} contains the main results of this paper: the general weighted discrete $p$-Hardy inequality with a sharp remainder (Theorem~\ref{thm hardy identity}), its various consequences and the quantitative stability result (Theorem~\ref{thm hardy stability}). In Section~\ref{sec proofs 1}, we prove the preliminary lemmata, namely the $C_p$-algebraic identity (Lemma~\ref{prop cp identity}) and the discrete critical Hardy inequality (Lemma~\ref{barki almost}). Section~\ref{sec proofs 2} is devoted to the proofs of the two main theorems. Finally, in Section~\ref{sec proofs 3}, we show how the general identity reduces to the Huang-Ye identity when $p=2$ (Corollary~\ref{cor 1}).

\section{Preliminaries}\label{sec pre}

In this section, we provide the necessary preliminary notation and results. First, we recall the definition of the $C_p$-functional and its properties. 

\begin{defn}
Let $1<p<\infty$ and $d\geq 1$. Then, for $\xi,\eta\in\mathbb{C}^{d}$, we define
\begin{align}\label{cp formula}
C_p(\xi,\eta):=|\xi|^p-|\xi-\eta|^p-p|\xi-\eta|^{p-2}\textnormal{Re}\langle \xi-\eta, \eta \rangle\geq0.
\end{align}
\end{defn}

One of the most important properties of the $C_p$-term is its lower bound for $p\geq 2$ by Cazacu, Krej{\v{c}}i{\v{r}}{\'\i}k, Lam and Laptev \cite{cazacu2024hardy}.

\begin{lem}[\text{\cite[Step 3 of Proof of Theorem 1.2]{cazacu2024hardy}}]\label{lem1}
Let $p\geq2$ and $d\geq 1$. Then, for $\xi,\eta\in\mathbb{C}^d$, we have
\begin{align*}
C_{p}(\xi,\eta)\geq c_1(p)|\eta|^{p},
\end{align*}
where 
\begin{align*}
c_1(p)
= \inf_{(s,t)\in\mathbb{R}^2\setminus\{(0,0)\}}
\frac{\bigl[t^2 + s^2 + 2s + 1\bigr]^{\frac p2} -1-ps}
{\bigl[t^2 + s^2\bigr]^{\frac p2}}\in(0,1].
\end{align*}
\end{lem}

The next property is concerned with the homogeneity of the function $C_p$. Recall that a general function $f$ of $n$ variables is called absolutely homogeneous of degree $k$ if
\begin{align*}
f(\lambda x_{1}, \ldots, \lambda x_{n}) = |\lambda|^{k}f(x_{1},\ldots, x_{n})
\end{align*}
for every $\lambda \neq 0$ and $x_{1}, \ldots, x_{n}$. The property below can be verified directly using the formula (\ref{cp formula}) and the definition of absolute homogeneity.

\begin{lem}
The $C_{p}$-functional is an absolutely homogeneous function of degree $p$. 
\end{lem}

Finally, the sharp remainder term of the algebraic inequality \cite[Lemma 2.6]{frank2008non} is established.

\begin{lem}\label{prop cp identity}
For any $a \in \C$ and $t \in [0,1]$, we have
\begin{align}\label{eq cp identity}
|a - t|^p - (1-t)^{p-1}\bigl(|a|^p - t\bigr) = C_p\bigl(a - t,\; t(a-1)\bigr) + t(1-t)^{p-1}\,C_p(1,\, 1-a).
\end{align}
\end{lem}

\begin{rem}\label{rem lem}
We observe that for general $1<p<\infty$, in (\ref{eq cp identity}), we require $0\leq t \leq 1$. However, when $p=2$, we can use the following identity 
\begin{align}\label{simplified}
|a-t|^{2}-(1-t)(|a|^{2}-t)=t|a-1|^{2}
\end{align}
to remove the condition $t \leq 1$. The identity (\ref{simplified}) can be checked by directly opening the brackets.
\end{rem}

We also note other interesting properties of the $C_p$-term in \cite{CT24, huang2025lp, yessirkegenov2026stability}. The final result we will need is due to Barki \cite[Proposition 4.7]{barki2024sharp} or Huang and Ye \cite[Theorem 1.4]{huang2024onedimensional}, where the discrete version of the critical Hardy inequality was proved. For our purposes, however, we provide a slightly improved version for all real-valued $u\in C_{c}(\mathbb{N})$ (see Section \ref{sec proofs 1} and (\ref{discrete hardy form})).

\begin{lem}\label{barki almost}
Let $1<p<\infty$. Then, for all real-valued $u\in C_{c}(\mathbb{N})$, we have
\begin{align}\label{barki almost eq}
\sum_{n=2}^{\infty} |u(n) - u(n-1)|^{p} \, n^{p-1}
\geq \left(\frac{p-1}{p}\right)^{p} \inf_{c \in \R}\sum_{n=2}^{\infty} \frac{|u(n)-c|^{p}}{n \log^{p}(n)}.
\end{align}
\end{lem}

\section{Main results}\label{sec main}

In this section, we present the main results of this paper, namely the sharp form of the general weighted discrete $p$-Hardy inequality and the stability results. In between, we also apply our results to obtain old and new discrete Hardy inequalities and identities.

\begin{thm}\label{thm hardy identity}
Let $1 < p < \infty$, $v$ be a non-negative function on $\N$ and $w$ be a real-valued function on $\mathbb{N}$. Let $\varphi(n+1)\geq\varphi(n)>0$ for all $n\in \mathbb{N}$ with $\varphi(0)=0$ and
\begin{align}\label{cond}
v(n)\bigl(\varphi(n) - \varphi(n-1)\bigr)^{p-1} - v(n+1)\bigl(\varphi(n+1) - \varphi(n)\bigr)^{p-1}\geq w(n)\varphi(n)^{p-1}.
\end{align}
Then, for every complex-valued $u \in C_c(\N_0)$ with $u(0) = 0$, we have
\begin{align}\label{eq hardy identity}
\sum_{n=1}^{\infty} v(n)|u(n) - u(n-1)|^p
\geq \sum_{n=1}^{\infty} w(n)|u(n)|^p + \sum_{n=2}^{\infty} v(n)R_p(u(n),\varphi(n)),
\end{align}
where at each index $n\geq 2$, we have
\begin{multline}\label{eq remainder}
R_p(u(n),\varphi(n))
:= C_p\!\left(u(n) - u(n-1),\;\frac{\varphi(n-1)}{\varphi(n)}\,u(n) - u(n-1)\right) \\ + \varphi(n-1)\bigl(\varphi(n) - \varphi(n-1)\bigr)^{p-1}\,C_p\!\left(\frac{u(n-1)}{\varphi(n-1)},\;\frac{u(n-1)}{\varphi(n-1)} - \frac{u(n)}{\varphi(n)}\right).
\end{multline}
Consequently, if $v(n)>0$ for every $n\geq2$, then the remainder term in (\ref{eq hardy identity}) vanishes if and only if
$
u(n)=c\phi(n)
$
for every $n\in\mathbb{N}$ and some $c\in\mathbb{C}$. Moreover, we have equality in (\ref{eq hardy identity}) if (\ref{cond}) is satisfied with equality.
\end{thm}

A few immediate remarks must be made about the theorem above. To begin with, after dropping the remainder term and rewriting the weight $w(n)$ from the condition (\ref{cond}), we obtain the general weighted $p$-Hardy inequality by {\v{S}}tampach and Waclawek \cite[Proposition 5]{stampach2026hardy}.

\begin{cor}
Let $1<p<\infty$, $v(n) \geq 0$ for all $n \in \mathbb{N}$, $\varphi(n+1) \geq \varphi(n) > 0$
for all $n \in \mathbb{N}$, and $\varphi(0) = 0$. Then for any complex-valued $u \in C_c(\mathbb{N}_0)$
with $u(0) = 0$, we have the inequality
\begin{align*}
  \sum_{n=1}^{\infty} v(n) \, |\nabla u(n)|^{p}
  \geq - \sum_{n=1}^{\infty} \frac{\operatorname{div}(v(n)(\nabla \varphi(n))^{p-1})}{\varphi(n)^{p-1}} \, |u(n)|^{p},
\end{align*}
where $\nabla \varphi(n) := \varphi(n) - \varphi(n-1)$ and
\begin{align*}
  \operatorname{div} h(n) := h(n+1) - h(n).
\end{align*}
\end{cor}

Then, if we specify $p=2$ in (\ref{eq hardy identity}), using the observation discussed in Remark \ref{rem lem}, we can recover the identity of Huang and Ye \cite{huang2024onedimensional} on $\mathbb{N}$. 

\begin{cor}\label{cor 1}
Let $v$ be a non-negative function on $\mathbb{N}$ and $\varphi: \mathbb{N}_{0}\to [0,\infty)$ be such that $\varphi(n) > 0$
for all $n \geq 1$ and $\varphi(0)=0$. Then for all complex-valued $u \in C_c(\mathbb{N}_{0})$ with $u(0)=0$, we have the identity
\begin{multline*}
  \sum_{n=1}^{\infty} v(n) |u(n)-u(n-1)|^2
  + \sum_{n=1}^{\infty} \frac{\operatorname{div}(v(n)(\nabla \varphi(n)))}{\varphi(n)} |u(n)|^2
  \\= \sum_{n=1}^{\infty} v(n+1) \left|
    \sqrt{\frac{\varphi(n)}{\varphi(n+1)}} \, u(n+1)
    - \sqrt{\frac{\varphi(n+1)}{\varphi(n)}} \, u(n)
  \right|^2,
\end{multline*}
where $\nabla \varphi(n) := \varphi(n) - \varphi(n-1)$ and
\begin{align*}
\operatorname{div} h(n) := h(n+1) - h(n).
\end{align*}
\end{cor}

Now we show how, by choosing the appropriate weights $v(n)$ and $\varphi(n)$ in the general result, we are able to recover several existing and obtain new discrete Hardy inequalities. To demonstrate, if we take
\begin{multline*}
v(n) := \begin{cases}
  0 & \text{if } n = 1, \\
  (n-1)^{\alpha} & \text{if } n \geq 2,
\end{cases}
\\
\varphi(n)=\hat{\varphi}_{Cop}(n) := \begin{cases}
  0 & \text{if } n = 0, \\
  \varphi_{Cop}(n+1) & \text{if } n \geq 1,
\end{cases}
\end{multline*}
where $\varphi_{Cop}(n) = \Gamma\!\left(n + 1 - \frac{\alpha+1}{p}\right) / \Gamma(n)$ and $\Gamma$ is the Euler Gamma function, we derive the next Copson inequality.

\begin{cor}\label{cor copson}
Let $1<p<\infty$, $\alpha<0$ and $u \in C_{c}(\mathbb{N}_0)$ be a complex-valued compactly supported sequence with $u(0)=u(1)=0$, then
\begin{multline}\label{copson ineq}
\sum_{n=1}^{\infty} n^{\alpha} |u(n)-u(n-1)|^{p}
\geq \left(\frac{p - \alpha - 1}{p}\right)^{p} \sum_{n=1}^{\infty} (n+1)^{\alpha-p} |u(n)|^{p}\\+\sum_{n=2}^{\infty}n^{\alpha}R_{p}\left(u(n), \hat{\varphi}_{Cop}(n)\right),
\end{multline}
where the constant $\left(\frac{p - \alpha - 1}{p}\right)^{p}$ is optimal. 
\end{cor}

In the same spirit as in \cite{gupta2022discrete}, we can choose the power weights
\begin{align*}
v(n)=n^{\alpha} \quad \text{and} \quad \varphi(n)=\varphi_{\beta}(n)=n^{\beta} \quad \text{for} \quad \alpha \in \mathbb{R} \quad \text{and} \quad \beta> 0
\end{align*}
to obtain the sharp remainder of the Fischer-Keller-Pogorzelski inequality. In fact, we get an exact identity.
\begin{cor}\label{cor gupta}
Let $1<p<\infty$, $\alpha\in\mathbb{R}$, $\beta > 0$ and $u\in C_{c}(\mathbb{N}_0)$ be a complex-valued compactly supported sequence with $u(0)=0$, then
\begin{align}\label{gen gupta main sec}
\sum_{n=1}^{\infty} n^{\alpha}|u(n) - u(n-1)|^p 
= \sum_{n=1}^{\infty}w_{p,\alpha,\beta}(n)|u(n)|^{p}+\sum_{n=2}^{\infty}n^{\alpha}R_{p}\left(u(n), \varphi_{\beta}(n)\right),
\end{align}
where
\begin{align*}
w_{p,\alpha,\beta}(n)=n^{\alpha} \left[\left(1 - \left(1 - \frac{1}{n}\right)^{\beta}\right)^{p-1}
- \left(1 + \frac{1}{n}\right)^{\alpha} \left(\left(1 + \frac{1}{n}\right)^{\beta} - 1\right)^{p-1}\right]
\end{align*}
for $n\geq2$, and $w_{p,\alpha,\beta}(1):=1-2^{\alpha}\left(2^{\beta}-1\right)^{p-1}$. Moreover, when $p=2$, identity \eqref{gen gupta main sec} extends to $\beta\in\mathbb{R}$.
\end{cor}
\begin{rem}
In the case $p=2$, dropping the last remainder term in the identity \eqref{gen gupta main sec} implies \cite[Theorem 2.1]{gupta2022discrete}. The case $\alpha=0$ and $\beta=\frac{p-1}{p}$ corresponds to the sharp form of the Fischer-Keller-Pogorzelski inequality (\ref{FKP}). Combining those two cases (i.e., $p=2$, $\alpha=0$ and $\beta=\frac{p-1}{p}$), we recover the Krej{\v{c}}i{\v{r}}{\'\i}k and \v{S}tampach identity \cite[Theorem 1]{krejcirik2022sharp}. Furthermore, by setting $\beta=-\frac{\alpha-p+1}{p}$, we obtain sharp remainder of the Barki result \cite[Theorem 4.11]{barki2024sharp} with restriction of $0\leq \alpha<p-1$.
\end{rem}

Finally, in the next result, we present the quantitative form of the discrete $p$-Hardy inequality.

\begin{thm}\label{thm hardy stability}
Let $p\geq2$. Then, for all real-valued $u \in C_c(\N_0)$ with $u(0)=0$, we have
  \begin{multline*}
    \sum_{n=1}^{\infty} |u(n) - u(n-1)|^p
    - \left(\frac{p-1}{p}\right)^p \sum_{n=1}^{\infty} \frac{|u(n)|^p}{n^p}
    \\\geq c_{1}(p)\left(\frac{p-1}{p}\right)^{p} \inf_{c \in \R} \sum_{n=3}^{\infty}
      \frac{\left|u(n) - c\,n^{\frac{p-1}{p}}\right|^p}{n^p\log^{p} (n)}.
  \end{multline*}
\end{thm}

When $p=2$, we get that $c_{1}(2)=1$, which gives the stability of the classical discrete Hardy inequality.

\begin{cor}
For all real-valued $u\in C_{c}(\mathbb{N}_{0})$ with $u(0)=0$, we have
  \begin{align*}
    \sum_{n=1}^{\infty} |u(n) - u(n-1)|^2
    - \frac{1}{4} \sum_{n=1}^{\infty} \frac{|u(n)|^2}{n^2}\geq \frac{1}{4}\inf_{c \in \R} \sum_{n=3}^{\infty}
      \frac{\left|u(n) - c\,\sqrt{n}\right|^2}{n^2\log^{2} (n)}.
  \end{align*}
\end{cor}

\section{Preliminary results--Proofs of Lemma \ref{prop cp identity} and Lemma \ref{barki almost}}\label{sec proofs 1}
\begin{proof}[\textbf{Proof of Lemma \ref{prop cp identity}}]
We evaluate the right-hand side of (\ref{eq cp identity}) and show it equals the left-hand side. Setting $\xi = a - t$ and $\eta = t(a-1)$ in the first $C_p$-term, we compute
\begin{align*}
\xi - \eta = (a - t) - t(a - 1) = a(1-t),
\end{align*}
so $|\xi - \eta| = (1-t)|a|$ since $t \in [0,1]$. The product evaluates as
\begin{align*}
\re\bigl(a(1-t)\cdot t(\bar{a}-1)\bigr)
= t(1-t)\bigl(|a|^2 - \re(a)\bigr),
\end{align*}
and substituting into the definition of $C_p$ yields
\begin{align}\label{eq term1}
C_p(a-t,\;t(a-1))
= |a - t|^p - (1-t)^p|a|^p - p\,t(1-t)^{p-1}|a|^{p-2}\bigl(|a|^2 - \re(a)\bigr).
\end{align}
For the second term, we set $\xi = 1$ and $\eta = 1 - a$, giving $\xi - \eta = a$ and
\begin{align*}
\re\bigl(a(1-\bar{a})\bigr) = \re(a) - |a|^2,
\end{align*}
so that
\begin{align*}
C_p(1,\;1-a) = 1 - |a|^p + p\,|a|^{p-2}\bigl(|a|^2 - \re(a)\bigr).
\end{align*}
Multiplying by $t(1-t)^{p-1}$ gives
\begin{multline}\label{eq term2}
t(1-t)^{p-1}\,C_p(1,1-a)
= t(1-t)^{p-1} - t(1-t)^{p-1}|a|^p\\ + p\,t(1-t)^{p-1}|a|^{p-2}\bigl(|a|^2 - \re(a)\bigr).
\end{multline}
Adding (\ref{eq term1}) and (\ref{eq term2}), the terms involving $p\,t(1-t)^{p-1}|a|^{p-2}(|a|^2 - \re(a))$ cancel exactly, leaving
\begin{align*}
\text{RHS}
= |a - t|^p - (1-t)^p|a|^p + t(1-t)^{p-1} - t(1-t)^{p-1}|a|^p.
\end{align*}
Grouping the terms containing $|a|^p$ and factoring, we observe
\begin{align*}
(1-t)^p + t(1-t)^{p-1} = (1-t)^{p-1}\bigl((1-t) + t\bigr) = (1-t)^{p-1},
\end{align*}
so the right-hand side simplifies to
\begin{align*}
\text{RHS} = |a-t|^p - (1-t)^{p-1}\bigl(|a|^p - t\bigr),
\end{align*}
which is exactly the left-hand side of (\ref{eq cp identity}).
\end{proof}

\begin{proof}[\textbf{Proof of Lemma \ref{barki almost}}]
First, we reduce the problem to sequences that vanish at $n=1$. Let us take a real-valued $u \in C_c(\mathbb{N})$. Since the right-hand side takes the infimum over all $c \in \mathbb{R}$, this infimum is bounded from above by evaluating the sum at the specific admissible choice $c = u(1)$. Let us define the shifted sequence $g(n) := u(n) - u(1)$ for $n \ge 1$. From the definition, $g(1) = 0$. Because $u$ has finite support, $g(n)$ is eventually constant (equal to $-u(1)$ for large $n$). Since $\sum_{n=2}^\infty \frac{1}{n \log^p n} < \infty$ for $p>1$, the series converges. Therefore, it is sufficient to prove:
\begin{align}\label{reduced eq}
\sum_{n=2}^{\infty} n^{p-1} |g(n) - g(n-1)|^{p} \geq \left(\frac{p-1}{p}\right)^{p} \sum_{n=2}^{\infty} \frac{|g(n)|^{p}}{n \log^{p}(n)}.
\end{align}
To establish \eqref{reduced eq}, we utilize the discrete version of the Picone inequality for the $p$-Laplacian (see \cite{amghibech2008discrete, brasco2014convexity, giacomoni2022discrete}): For any real $s, t \ge 0$ and any $A > B > 0$, we have
\begin{align}\label{picone}
|s - t|^p \geq s^p \left(\frac{A-B}{A}\right)^{p-1} - t^p \left(\frac{A-B}{B}\right)^{p-1}.
\end{align}
Let $\varphi(n)$ be a strictly increasing positive sequence for $n \ge 2$, with the boundary condition $\varphi(1) = 0$. For $n \ge 3$, we substitute $s = |g(n)|$, $t = |g(n-1)|$, $A = \varphi(n)$, and $B = \varphi(n-1)$. By the reverse triangle inequality, we know that 
\begin{align*}
|g(n) - g(n-1)| \geq \big| |g(n)| - |g(n-1)| \big|.
\end{align*}
Raising both sides to the power of $p$, we obtain $|g(n) - g(n-1)|^p \ge |s - t|^p$. Now using the Picone inequality (\ref{picone}) and multiplying by $n^{p-1}$, we get
\begin{align}\label{main ineq}
n^{p-1} |g(n) - g(n-1)|^p \geq |g(n)|^p \frac{n^{p-1}(\nabla \varphi(n))^{p-1}}{\varphi(n)^{p-1}} - |g(n-1)|^p \frac{n^{p-1}(\nabla \varphi(n))^{p-1}}{\varphi(n-1)^{p-1}},
\end{align}
where $\nabla \varphi(n) := \varphi(n) - \varphi(n-1)$. For $n=2$, the Picone inequality (\ref{picone}) cannot be applied directly because $\varphi(1)=0$. Nevertheless, since $g(1)=0$, the corresponding estimate is verified directly as the equality: 
\begin{align*}
2^{p-1}|g(2)-g(1)|^p = 2^{p-1}|g(2)|^p = |g(2)|^p \frac{2^{p-1}(\varphi(2)-\varphi(1))^{p-1}}{\varphi(2)^{p-1}}.
\end{align*}
Let us define the weight $h(n) := n^{p-1}(\nabla \varphi(n))^{p-1}$. Summing the inequality (\ref{main ineq}) from $n=2$ to $M$ gives
\begin{align}\label{discrete hardy form M}
\sum_{n=2}^M n^{p-1} |\nabla g(n)|^p \geq \sum_{n=2}^{M-1} |g(n)|^p \left[ \frac{h(n) - h(n+1)}{\varphi(n)^{p-1}} \right] + |g(M)|^p \frac{h(M)}{\varphi(M)^{p-1}}.
\end{align}
To achieve the inequality (\ref{reduced eq}), we now explicitly construct $h(n)$. Let $\beta = \frac{p-1}{p}$. We define $h(n)$ for $n \ge 2$ as:
\begin{align*}
h(n) := \beta^{p-1} \left( \log\left(n - \frac{1}{2}\right) \right)^{-\beta}.
\end{align*}
We define our supersolution sequence $\varphi(n)$ recursively so that its gradient matches $h(n)$. Setting $\varphi(1) = 0$, for $n \ge 2$ we define:
\begin{align*}
\nabla \varphi(n) = \left( \frac{h(n)}{n^{p-1}} \right)^{\frac{1}{p-1}} = \frac{\beta}{n \left( \log\left(n - \frac{1}{2}\right) \right)^{1/p}}.
\end{align*}
Thus, for all $n \ge 2$, $\varphi(n)$ is uniquely determined as:
\begin{align*}
\varphi(n) = \sum_{k=2}^n \frac{\beta}{k \left( \log\left(k - \frac{1}{2}\right) \right)^{1/p}}.
\end{align*}
We must bound $\varphi(n)$ strictly from above. Let $F(x) = \frac{1}{x (\log x)^{1/p}}$. 
Computing the second derivative of $F(x)$ gives:
\begin{align*}
F''(x) = \frac{1}{x^3 (\log x)^{1/p+2}} \left[ 2(\log x)^2 + \frac{3}{p} \log x + \frac{1}{p}\left(1+\frac{1}{p}\right) \right].
\end{align*}
For $x > 1$, $\log x > 0$, ensuring $F''(x) > 0$. Thus, $F(x)$ is strictly convex on $(1, \infty)$. By the Hermite-Hadamard inequality, for any strictly convex function on $[k-1, k]$, the value at the midpoint $k - 1/2$ is strictly less than the integral average. Utilizing the fact that $\frac{1}{k} < \frac{1}{k - 1/2}$, we obtain the strict bound for $k \ge 3$:
\begin{align*}
\frac{1}{k \left( \log\left(k - \frac{1}{2}\right) \right)^{1/p}} < \frac{1}{\left(k - \frac{1}{2}\right) \left( \log\left(k - \frac{1}{2}\right) \right)^{1/p}} = F\left(k - \frac{1}{2}\right) < \int_{k-1}^k F(x) dx.
\end{align*}
For the singular interval at $k=2$, we bound the integral below by the tangent line to $F(x)$ at $x=1.5$. By strict convexity, $F(x) > F(1.5) + F'(1.5)(x - 1.5)$ a.e. on $x \in (1, 2]$. Integrating this lower bound over $(1, 2]$ yields exactly the midpoint value $F(1.5)$:
\begin{align*}
\int_{1}^{2}F(x)dx > \int_{1}^{2}\left[F(1.5) + F'(1.5)(x - 1.5)\right]dx=F(1.5).
\end{align*}
Thus,
\begin{align*}
\frac{1}{2 \left( \log 1.5 \right)^{1/p}} < \frac{1}{1.5 \left( \log 1.5 \right)^{1/p}} = F(1.5) < \int_1^2 F(x) \, dx.
\end{align*}
Summing these bounds from $k=2$ to $n$:
\begin{align*}
\varphi(n) < \sum_{k=2}^n \int_{k-1}^k \frac{\beta}{x (\log x)^{1/p}} \, dx = \int_1^n \frac{\beta}{x (\log x)^{1/p}} \, dx.
\end{align*}
Because $\int \beta x^{-1}(\log x)^{-1/p} \, dx = (\log x)^{1 - 1/p} = (\log x)^\beta$, we obtain the exact upper bound:
\begin{align}\label{phi bound}
\varphi(n) < (\log n)^\beta \implies \varphi(n)^{p-1} < (\log n)^{\beta(p-1)} = (\log n)^{\frac{(p-1)^2}{p}}.
\end{align}
Next, we establish a rigorous lower bound for the strictly positive difference $h(n) - h(n+1)$:
\begin{align}\label{hn-hn+1}
h(n) - h(n+1) &= \beta^{p-1} \left[ \left( \log\left(n - \frac{1}{2}\right) \right)^{-\beta} - \left( \log\left(n + \frac{1}{2}\right) \right)^{-\beta} \right].
\end{align}
By the fundamental theorem of calculus, we have
\begin{align*}
\int_{a}^{b}f'(x)dx=f(b)-f(a).
\end{align*}
Taking $f:=(\log x)^{-\beta}$, we get
\begin{align*}
f(n-1/2) - f(n+1/2) = -\int_{n-1/2}^{n+1/2} \left( -\frac{\beta}{x(\log x)^{\beta+1}} \right) dx = \int_{n-1/2}^{n+1/2} \frac{\beta}{x(\log x)^{\beta+1}}dx,
\end{align*}
which means that
\begin{align*}
\left( \log\left(n - \frac{1}{2}\right) \right)^{-\beta} - \left( \log\left(n + \frac{1}{2}\right) \right)^{-\beta} = \int_{n-1/2}^{n+1/2} \frac{\beta}{x(\log x)^{\beta+1}}dx.
\end{align*}
Using this in (\ref{hn-hn+1}), we obtain
\begin{align*}
h(n) - h(n+1) = \beta^{p-1} \int_{n-1/2}^{n+1/2} \frac{\beta}{x (\log x)^{\beta+1}} dx = \beta^p \int_{n-1/2}^{n+1/2} \frac{1}{x (\log x)^{\beta+1}} dx.
\end{align*}
Let $G(x) = \frac{1}{x (\log x)^{\beta+1}}$. Computing its second derivative gives:
\begin{align*}
G''(x) = \frac{1}{x^3 (\log x)^{\beta+3}} \left[ 2(\log x)^2 + 3(\beta+1) \log x + (\beta+1)(\beta+2) \right].
\end{align*}
Since $\beta > 0$ and $\log x > 0$ for $x > 1$, we have $G''(x) > 0$, making $G(x)$ strictly convex on $(1, \infty)$. Applying the Hermite-Hadamard inequality over $[n - \frac{1}{2}, n + \frac{1}{2}]$ (where the midpoint is exactly $n$), the integral strictly dominates the midpoint evaluation:
\begin{align*}
\int_{n-1/2}^{n+1/2} G(x) \, dx > G(n) = \frac{1}{n (\log n)^{\beta+1}}.
\end{align*}
Since $\beta + 1 = \frac{p-1}{p} + 1 = \frac{2p-1}{p}$, we obtain the exact lower bound:
\begin{align*}
h(n) - h(n+1) > \frac{\beta^p}{n (\log n)^{\frac{2p-1}{p}}}.
\end{align*}
Now we combine our bounds. Since $\varphi(n)^{p-1} < (\log n)^{\frac{(p-1)^2}{p}}$ by (\ref{phi bound}), we have:
\begin{align*}
\frac{h(n) - h(n+1)}{\varphi(n)^{p-1}} &> \frac{\frac{\beta^p}{n (\log n)^{\frac{2p-1}{p}}}}{(\log n)^{\frac{(p-1)^2}{p}}} = \frac{\beta^p}{n (\log n)^{\frac{2p-1}{p} + \frac{(p-1)^2}{p}}}.
\end{align*}
Simplifying the exponent of the logarithm gives $\frac{2p-1 + p^2 - 2p + 1}{p} = \frac{p^2}{p} = p$. Thus, for every $n \ge 2$, we have that
\begin{align*}
\frac{h(n) - h(n+1)}{\varphi(n)^{p-1}} > \left(\frac{p-1}{p}\right)^p \frac{1}{n \log^p n}.
\end{align*}
Substituting this bound into \eqref{discrete hardy form M} gives:
\begin{align}\label{huang-ye any p M}
\sum_{n=2}^{M} n^{p-1} |g(n) - g(n-1)|^p \geq \left(\frac{p-1}{p}\right)^{p} \sum_{n=2}^{M-1} \frac{|g(n)|^{p}}{n \log^{p}(n)} + |g(M)|^{p}\frac{h(M)}{\varphi(M)^{p-1}}.
\end{align}
Because $h(M) > 0$ and $\varphi(M) > 0$, the boundary term at $M$ is non-negative and can be safely dropped. Letting $M \to \infty$ gives:
\begin{align}\label{discrete hardy form}
\sum_{n=2}^\infty n^{p-1} |g(n)-g(n-1)|^p \geq \left(\frac{p-1}{p}\right)^{p} \sum_{n=2}^{\infty} \frac{|g(n)|^{p}}{n \log^{p}(n)}.
\end{align}
Since $g(n) := u(n)-u(1)$ and taking the infimum of the right-hand side over all possible constants $c \in \mathbb{R}$ only creates a value less than or equal to the evaluation at $c = u(1)$, we finally get
\begin{align*}
\sum_{n=2}^{\infty} |u(n) - u(n-1)|^{p} n^{p-1}  \geq \left(\frac{p-1}{p}\right)^{p} \inf_{c \in \mathbb{R}}\sum_{n=2}^{\infty} \frac{|u(n)-c|^{p}}{n \log^{p}(n)},
\end{align*}
which completes the proof.
\end{proof}

\section{The general weighted discrete $p$-Hardy inequality and Stability--Proofs of Theorem \ref{thm hardy identity} and Theorem \ref{thm hardy stability}}\label{sec proof hardy}\label{sec proofs 2}

\begin{proof}[\textbf{Proof of Theorem \ref{thm hardy identity}}]
Let $X, Y \in \C$ with $Y \neq 0$ and $t \in [0,1]$. Substituting $a = X/Y$ into (\ref{eq cp identity}) and multiplying by $|Y|^p$, we use the $p$-homogeneity of $C_p$ to obtain the exact identity
\begin{multline}\label{eq homogenized}
|X - tY|^p
= (1-t)^{p-1}\bigl(|X|^p - t|Y|^p\bigr)
+ C_p\bigl(X - tY,\;t(X-Y)\bigr)
\\+ t(1-t)^{p-1}\,C_p(Y,\;Y-X).
\end{multline}
Indeed, the left-hand side of (\ref{eq cp identity}) becomes $|X - tY|^p - (1-t)^{p-1}(|X|^p - t|Y|^p)$ after scaling, while on the right-hand side homogeneity converts $C_p(X/Y - t,\;t(X/Y - 1))$ into $|Y|^{-p}\,C_p(X - tY,\;t(X-Y))$ and $C_p(1,\;1-X/Y)$ into $|Y|^{-p}\,C_p(Y,\;Y-X)$. Thus, for $Y \neq 0$, we obtain \eqref{eq homogenized}. If $Y=0$, identity \eqref{eq homogenized} follows directly from the definition of $C_p$ (\ref{cp formula}). Hence it actually holds for all $X, Y \in \mathbb{C}$. Now define $\psi(n) := u(n)/\varphi(n)$ for $n \in \N$ and $\psi(0) := 0$. Since $\varphi$ is positive and non-decreasing on $\N$, the ratio $t_n := \varphi(n-1)/\varphi(n)$ lies in $[0,1]$ for every $n \geq 1$. Applying (\ref{eq homogenized}) with $X = \psi(n)$, $Y = \psi(n-1)$ and $t = t_n$ for $n\geq2$, and treating the case
$n=1$ separately using $\psi(0)=0$ and $t_1=0$, then multiplying the entire identity by $\varphi(n)^p$, the left-hand side becomes
\begin{align*}
\varphi(n)^p\left|\psi(n) - \frac{\varphi(n-1)}{\varphi(n)}\psi(n-1)\right|^p
= |u(n) - u(n-1)|^p.
\end{align*}
The first term on the right yields
\begin{align*}
\bigl(\varphi(n) - \varphi(n-1)\bigr)^{p-1}\Bigl(\varphi(n)\,|\psi(n)|^p - \varphi(n-1)\,|\psi(n-1)|^p\Bigr).
\end{align*}
For the two $C_p$ terms, distributing $\varphi(n)$ into the arguments via homogeneity produces exactly $R_p(u(n),\varphi(n))$ as defined in (\ref{eq remainder}). To see this, we first work with the first $C_p$ term:
\begin{align*}
\varphi(n)^p &\,C_p\bigl(\psi(n) - t_n\psi(n-1),\; t_n(\psi(n) - \psi(n-1))\bigr) \\
&= C_p\bigl(\varphi(n)\psi(n) - \varphi(n)t_n\psi(n-1),\; \varphi(n)t_n(\psi(n) - \psi(n-1))\bigr) \\
&= C_p\!\left(u(n) - \varphi(n-1)\frac{u(n-1)}{\varphi(n-1)},\;
   \varphi(n-1)\!\left(\frac{u(n)}{\varphi(n)} - \frac{u(n-1)}{\varphi(n-1)}\right)\right) \\
&= C_p\!\left(u(n) - u(n-1),\;
   \frac{\varphi(n-1)}{\varphi(n)}\,u(n) - u(n-1)\right).
\end{align*}
For the second term:
\begin{align*}
\varphi(n)^p &\,t_n(1-t_n)^{p-1}\,C_p\bigl(\psi(n-1),\; \psi(n-1) - \psi(n)\bigr) \\
&= \varphi(n)^p \left(\frac{\varphi(n-1)}{\varphi(n)}\right)
   \left(1 - \frac{\varphi(n-1)}{\varphi(n)}\right)^{p-1}
   C_p\!\left(\frac{u(n-1)}{\varphi(n-1)},\;
   \frac{u(n-1)}{\varphi(n-1)} - \frac{u(n)}{\varphi(n)}\right) \\
&= \varphi(n-1)\bigl(\varphi(n) - \varphi(n-1)\bigr)^{p-1}\,
   C_p\!\left(\frac{u(n-1)}{\varphi(n-1)},\;
   \frac{u(n-1)}{\varphi(n-1)} - \frac{u(n)}{\varphi(n)}\right).
\end{align*}
Assembling, we arrive at the identity
\begin{multline}\label{eq pointwise}
|u(n) - u(n-1)|^p
= \bigl(\varphi(n) - \varphi(n-1)\bigr)^{p-1}\Bigl(\varphi(n)\,|\psi(n)|^p - \varphi(n-1)\,|\psi(n-1)|^p\Bigr) \\+ R_p(u(n),\varphi(n))
\end{multline}
for all $n\in \mathbb{N}$, where $\psi(0):=0$ and $R_p(u(1),\varphi(1)):=0$. Multiplying (\ref{eq pointwise}) by $v(n)$ and summing over $n$, it remains to evaluate the sum
\begin{align}\label{eq before split}
S := \sum_{n=1}^{\infty} v(n)\bigl(\varphi(n) - \varphi(n-1)\bigr)^{p-1}\Bigl(\varphi(n)\,|\psi(n)|^p - \varphi(n-1)\,|\psi(n-1)|^p\Bigr).
\end{align}
Writing $\alpha(n) := v(n)(\varphi(n) - \varphi(n-1))^{p-1}$ for brevity, we split
\begin{align}\label{eq split}
S = \sum_{n=1}^{\infty} \alpha(n)\,\varphi(n)\,|\psi(n)|^p
\;-\; \sum_{n=1}^{\infty} \alpha(n)\,\varphi(n-1)\,|\psi(n-1)|^p.
\end{align}
Note that since $u(n)$ is in $C_{c}(\mathbb{N}_0)$ and $\psi(n) = u(n)/\varphi(n)$, both series in (\ref{eq before split}) and (\ref{eq split}) converge. In the second sum, shifting the index by $m = n - 1$ and using that $\psi(0) = 0$ and $u$ has finite support (so all boundary terms vanish) gives
\begin{align*}
\sum_{n=1}^{\infty} \alpha(n)\,\varphi(n-1)\,|\psi(n-1)|^p
= \sum_{n=1}^{\infty} v(n+1)\bigl(\varphi(n+1) - \varphi(n)\bigr)^{p-1}\varphi(n)\,|\psi(n)|^p.
\end{align*}
Combining and factoring out $\varphi(n)\,|\psi(n)|^p = |u(n)|^p / \varphi(n)^{p-1}$ yields
\begin{align*}
S = \sum_{n=1}^{\infty}\frac{v(n)\bigl(\varphi(n) - \varphi(n-1)\bigr)^{p-1} - v(n+1)\bigl(\varphi(n+1) - \varphi(n)\bigr)^{p-1}}{\varphi(n)^{p-1}}\;|u(n)|^p.
\end{align*}
Finally, using the assumption \eqref{cond}:
\begin{align}\label{cond proof}
v(n)\bigl(\varphi(n) - \varphi(n-1)\bigr)^{p-1} - v(n+1)\bigl(\varphi(n+1) - \varphi(n)\bigr)^{p-1}\geq w(n)\varphi(n)^{p-1},
\end{align}
we obtain 
\begin{align}\label{identity proof}
\sum_{n=1}^{\infty} v(n)|u(n) - u(n-1)|^p 
\geq \sum_{n=1}^{\infty} w(n)|u(n)|^p + \sum_{n=2}^{\infty} R_p(u(n),\varphi(n))v(n).
\end{align}
Since the identity is preserved throughout until assumption (\ref{cond proof}) is used, it follows that if (\ref{cond proof}) holds with equality, then so does (\ref{identity proof}). Also, since the $C_p$-functional is non-negative and the remainder $R_p(u(n),\varphi(n))$ is defined as a sum of two non-negative terms, it vanishes if and only if both $C_p$ terms are simultaneously zero. However, we know that $C_p(\xi, \eta) = 0$ if and only if $\eta = 0$. Setting the second argument of the second $C_p$ term in \eqref{eq remainder} to zero yields
\begin{align*}
\frac{u(n-1)}{\varphi(n-1)} - \frac{u(n)}{\varphi(n)} = 0,
\end{align*}
which implies that the ratio $\psi(n) = u(n)/\varphi(n)$ must be a constant $c \in \mathbb{C}$ for all $n \geq 1$. Thus, $u(n) = c\varphi(n)$. Substituting this into the second argument of the first $C_p$ term gives
\begin{align*}
\frac{\varphi(n-1)}{\varphi(n)} c \varphi(n) - c \varphi(n-1) = 0,
\end{align*}
which confirms that the first $C_p$ term also evaluates to zero. Consequently, the remainder term vanishes formally if and only if $u(n) = c\varphi(n)$. This completes the proof.
\end{proof}

\begin{proof}[\textbf{Proof of Theorem \ref{thm hardy stability}}]
Let us recall the discrete Hardy identity (\ref{gen gupta main sec}) that we have established through our general result:
\begin{align*}
\sum_{n=1}^{\infty} n^{\alpha}|u(n) - u(n-1)|^p 
= \sum_{n=1}^{\infty}w_{p,\alpha,\beta}(n)|u(n)|^{p}+\sum_{n=2}^{\infty}n^{\alpha}R_{p}\left(u(n), \varphi_{\beta}(n)\right),
\end{align*}
where
\begin{align*}
w_{p,\alpha,\beta}=n^{\alpha} \left[\left(1 - \left(1 - \frac{1}{n}\right)^{\beta}\right)^{p-1}
- \left(1 + \frac{1}{n}\right)^{\alpha} \left(\left(1 + \frac{1}{n}\right)^{\beta} - 1\right)^{p-1}\right]
\end{align*}
and
\begin{multline}\label{stability rp}
R_p(u(n),\varphi_{\beta}(n))
= C_p\!\left(u(n) - u(n-1),\;\frac{\varphi_{\beta}(n-1)}{\varphi_{\beta}(n)}\,u(n) - u(n-1)\right)  
\\ + \varphi_{\beta}(n-1)\bigl(\varphi_{\beta}(n) - \varphi_{\beta}(n-1)\bigr)^{p-1}\,C_p\!\left(\frac{u(n-1)}{\varphi_{\beta}(n-1)},\;\frac{u(n-1)}{\varphi_{\beta}(n-1)} - \frac{u(n)}{\varphi_{\beta}(n)}\right).
\end{multline}
As we noted before, taking $\alpha=0$ and $\beta=\frac{p-1}{p}$, we get the sharp remainder formula of the Fischer-Keller-Pogorzelski inequality:
\begin{align*}
\sum_{n=1}^{\infty} |u(n) - u(n-1)|^p 
= \sum_{n=1}^{\infty}w_{p}(n)|u(n)|^{p}+\sum_{n=2}^{\infty}R_{p}\left(u(n), \varphi_{p}(n)\right),
\end{align*}
where $\varphi_p(n) := n^{\frac{p-1}{p}}$. Utilizing the inequality (\ref{fkp ineq weight}), we obtain
\begin{multline}\label{last ser}
\sum_{n=1}^{\infty} |u(n) - u(n-1)|^{p}
- \left(\frac{p-1}{p}\right)^{p} \sum_{n=1}^{\infty} \frac{|u(n)|^{p}}{n^{p}}
\\= \sum_{n=1}^{\infty} \left(w_p(n) - w_p^{H}(n)\right) |u(n)|^{p}
+ \sum_{n=2}^{\infty} R_p(u(n), \varphi_{p}(n)).
\end{multline}
Now dropping the first series on the RHS of (\ref{last ser}), the last $C_{p}$-term in (\ref{stability rp}) and using Lemma \ref{lem1}, we get that for $p\geq 2$, we have
\begin{multline}\label{stab mid}
\sum_{n=1}^{\infty} |u(n) - u(n-1)|^{p}
- \left(\frac{p-1}{p}\right)^{p} \sum_{n=1}^{\infty} \frac{|u(n)|^{p}}{n^{p}}
\\\geq c_1(p) \sum_{n=2}^{\infty} (n-1)^{p-1} \left|\frac{u(n)}{n^{\frac{p-1}{p}}}-\frac{u(n-1)}{(n-1)^{\frac{p-1}{p}}}\right|^{p}.
\end{multline}
Let us define 
\begin{align}\label{defs}
S:=c_1(p) \sum_{n=2}^{\infty} (n-1)^{p-1} \left|f(n)-f(n-1)\right|^{p} \quad \text{with} \quad f(n):=\frac{u(n)}{n^{\frac{p-1}{p}}} \quad \text{and} \quad f(0):=0.
\end{align}
Let us isolate the first term of the series $S$:
\begin{align*}
S=c_1(p)|f(2)-f(1)|^{p}+c_1(p)\sum_{n=3}^{\infty} (n-1)^{p-1} \left|f(n)-f(n-1)\right|^{p}.
\end{align*}
Dropping the first non-negative term and shifting the index, we get
\begin{align}\label{s geq 1}
S\geq c_1(p)\sum_{n=2}^{\infty} n^{p-1} \left|f(n+1)-f(n)\right|^{p}.
\end{align}
Here, we define a new shifted sequence $g(n):=f(n+1)$ for $n\geq1$. Since $f$ has finite support, so does $g$. Applying Lemma \ref{barki almost} for the sequence $g$, we derive:
\begin{align*}
\sum_{n=2}^{\infty} n^{p-1} |g(n) - g(n-1)|^{p}
\geq \left(\frac{p-1}{p}\right)^{p} \inf_{c \in \mathbb{R}} \sum_{n=2}^{\infty} \frac{|g(n) - c|^{p}}{n \log^{p}(n)}.
\end{align*}
Substituting $g(n) = f(n+1)$ back, we get:
\begin{align}\label{stab last proof}
\sum_{n=2}^{\infty} n^{p-1} |f(n+1) - f(n)|^{p}
&\geq \left(\frac{p-1}{p}\right)^{p} \inf_{c \in \mathbb{R}} \sum_{n=2}^{\infty} \frac{|f(n+1) - c|^{p}}{n \log^{p}(n)} \nonumber
\\&\geq \left(\frac{p-1}{p}\right)^{p} \inf_{c \in \mathbb{R}} \sum_{n=2}^{\infty} \frac{|f(n+1) - c|^{p}}{(n+1) \log^{p}(n+1)}.
\end{align}
Putting (\ref{stab mid})-(\ref{stab last proof}) together, we obtain 
\begin{multline*}
\sum_{n=1}^{\infty} |u(n) - u(n-1)|^{p}
- \left(\frac{p-1}{p}\right)^{p} \sum_{n=1}^{\infty} \frac{|u(n)|^{p}}{n^{p}}\\\geq c_{1}(p)\left(\frac{p-1}{p}\right)^{p}\inf_{c \in \mathbb{R}} \sum_{n=2}^{\infty} \frac{\left|\frac{u(n+1)}{(n+1)^{\frac{p-1}{p}}} - c\right|^{p}}{(n+1) \log^{p}(n+1)}
\end{multline*}
Finally, simplifying and shifting the index one last time, we have
\begin{multline*}
\sum_{n=1}^{\infty} |u(n) - u(n-1)|^{p}
- \left(\frac{p-1}{p}\right)^{p} \sum_{n=1}^{\infty} \frac{|u(n)|^{p}}{n^{p}}\\\geq c_{1}(p)\left(\frac{p-1}{p}\right)^{p}\inf_{c \in \mathbb{R}} \sum_{n=3}^{\infty} \frac{\left|u(n)-c\,n^{\frac{p-1}{p}}\right|^{p}}{n^{p}\log^{p}(n)}.
\end{multline*}
The proof is complete.
\end{proof}

\section{Some consequences of the main results--Proofs of Corollaries \ref{cor 1}, \ref{cor copson} and \ref{cor gupta}}\label{sec proofs 3}

\begin{proof}[\textbf{Proof of Corollary \ref{cor 1}}]
When $p=2$, in the proof of Theorem \ref{thm hardy identity}, we can drop the non-decreasing assumption on $\varphi$ (see Remark \ref{rem lem}). Thus, under this substitution, we get
\begin{align}\label{eq proof cor 1}
\sum_{n=1}^{\infty} v(n)|u(n) - u(n-1)|^2
= \sum_{n=1}^{\infty} w(n)|u(n)|^2 + \sum_{n=2}^{\infty} v(n)R_2(u(n),\varphi(n)),
\end{align}
where we can take
\begin{align*}
w(n)=\frac{v(n)\bigl(\varphi(n) - \varphi(n-1)\bigr) - v(n+1)\bigl(\varphi(n+1) - \varphi(n)\bigr)}{\varphi(n)}.
\end{align*}
Since $C_{2}(\xi,\eta)=|\eta|^{2}$, we have
\begin{align*}
R_2(u(n),\varphi(n))
&= \left|\frac{\varphi(n-1)}{\varphi(n)}u(n) - u(n-1)\right|^{2}
\\& \quad + \varphi(n-1)\bigl(\varphi(n) - \varphi(n-1)\bigr)\left|\frac{u(n)}{\varphi(n)}-\frac{u(n-1)}{\varphi(n-1)}\right|^{2}
\\&  = (\varphi(n-1))^{2}\left|\frac{u(n)}{\varphi(n)}-\frac{u(n-1)}{\varphi(n-1)}\right|^{2}
\\& \quad  + \varphi(n-1)\bigl(\varphi(n) - \varphi(n-1)\bigr)\left|\frac{u(n)}{\varphi(n)}-\frac{u(n-1)}{\varphi(n-1)}\right|^{2}
\\& = \varphi(n-1)\varphi(n)\left|\frac{u(n)}{\varphi(n)}-\frac{u(n-1)}{\varphi(n-1)}\right|^{2}.
\end{align*}
Using the fact that $\varphi(0)=0$ and shifting the last series on the right-hand side of (\ref{eq proof cor 1}), we get 
\begin{align*}
\sum_{n=2}^{\infty} v(n)R_2(u(n),\varphi(n)) &=\sum_{n=2}^{\infty}v(n)\varphi(n-1)\varphi(n)\left|\frac{u(n)}{\varphi(n)}-\frac{u(n-1)}{\varphi(n-1)}\right|^{2}
\\&=\sum_{n=1}^{\infty}v(n+1)\varphi(n)\varphi(n+1)\left|\frac{u(n+1)}{\varphi(n+1)}-\frac{u(n)}{\varphi(n)}\right|^{2}
\\&=\sum_{n=1}^{\infty}v(n+1)\left|\sqrt{\frac{\varphi(n)}{\varphi(n+1)}}u(n+1)-\sqrt{\frac{\varphi(n+1)}{\varphi(n)}}u(n)\right|^{2}.
\end{align*}
This completes the proof.
\end{proof}

\begin{proof}[\textbf{Proof of Corollary \ref{cor copson}}]
Let
\begin{align*}
v(n) := \begin{cases}
  0, & n = 1, \\
  (n-1)^{\alpha}, & n \geq 2,
\end{cases}
\end{align*}
and
\begin{align*}
\varphi(n)=\varphi_{Cop}(n) := \begin{cases}
  0, & n = 0, \\
  \frac{\Gamma\left(n + 1 - \frac{\alpha+1}{p}\right)}{\Gamma(n)}, & n \geq 1.
\end{cases}
\end{align*}
Define the shifted sequence
\begin{align*}
    U(0) := 0, \quad U(n) := u(n-1), \quad n \ge 1.
\end{align*}
Since $u(0) = u(1) = 0$, we have
\begin{align*}
    U(0) = U(1) = U(2) = 0.
\end{align*}
Applying Theorem \ref{thm hardy identity} to $U$, $v$, and $\varphi_{Cop}$, we obtain
\begin{multline*}
    \sum_{n=2}^\infty (n-1)^\alpha |U(n) - U(n-1)|^p \ge \sum_{n=2}^\infty w(n) |U(n)|^p \\
    + \sum_{n=2}^\infty (n-1)^\alpha R_p(U(n), \varphi_{Cop}(n)),
\end{multline*}
where
\begin{align*}
    w(n) = \left( \frac{p-\alpha-1}{p} \right)^{p-1} n^{\alpha-p+1} H\left( \frac{1}{n} \right)
\end{align*}
and
\begin{align*}
    H(x) = (1-x)^\alpha \left( 1 - \frac{\alpha+1}{p}x \right)^{1-p} - 1.
\end{align*}
To obtain \eqref{copson ineq}, it is enough to show that
\begin{align}\label{cop proof}
    H(x) \geq \frac{p-\alpha-1}{p}x \quad \text{for every } x \in [0, 1).
\end{align}
Indeed, since
\begin{align*}
    H(0) = 0, \quad H'(0) = \frac{p-\alpha-1}{p},
\end{align*}
we have that \eqref{cop proof} is equivalent to $H(x) \geq H(0) + H'(0)x$. So, to show \eqref{cop proof}, we just show $H(x)$ is strictly convex on $(0,1)$. Calculating the second derivative, we get
\begin{multline*}
    H''(x) = \frac{p - \alpha - 1}{p^2} \frac{(1-x)^{\alpha-2}}{\left(1 - \frac{\alpha+1}{p}\,x\right)^{p+1}}
    \biggl[(p - \alpha)(1 + \alpha)^2 x^2 \\- 2(1 + \alpha)(p - \alpha)x + p(1 - \alpha)\biggr].
\end{multline*}
Since $\alpha < 0$, the discriminant
\[
    4\alpha(1+\alpha)^2(p-1)(p-\alpha)
\]
is non-positive. If $\alpha \neq -1$, it is strictly negative and the leading coefficient is positive. If $\alpha = -1$, the polynomial reduces to the positive constant $2p$. Hence the polynomial is positive for every $x \in (0,1)$. Therefore, $H(x)$ is strictly convex on $(0,1)$ and
\begin{multline*}
    \sum_{n=2}^\infty (n-1)^\alpha |U(n) - U(n-1)|^p \ge \left( \frac{p-\alpha-1}{p} \right)^p \sum_{n=2}^\infty n^{\alpha-p} |U(n)|^p \\
    + \sum_{n=2}^\infty (n-1)^\alpha R_p(U(n), \varphi_{Cop}(n)).
\end{multline*}
Now set $m = n-1$. Since 
\begin{align*}
    U(n) = u(n-1) = u(m), \quad U(n-1) = u(m-1),
\end{align*}
we have
\begin{align*}
    \sum_{n=2}^\infty (n-1)^\alpha |U(n) - U(n-1)|^p = \sum_{m=1}^\infty m^\alpha |u(m) - u(m-1)|^p,
\end{align*}
and
\begin{align*}
    \sum_{n=2}^\infty n^{\alpha-p} |U(n)|^p = \sum_{m=1}^\infty (m+1)^{\alpha-p} |u(m)|^p.
\end{align*}
Define
\begin{align*}
    \hat{\varphi}_{Cop}(0) := 0, \quad \hat{\varphi}_{Cop}(m) := \varphi_{Cop}(m+1), \quad m \ge 1.
\end{align*}
For $n \ge 3$, the definition of $R_p$ gives
\begin{align*}
    R_p(U(n), \varphi_{Cop}(n)) = R_p(u(m), \hat{\varphi}_{Cop}(m)).
\end{align*}
The term corresponding to $n=2$ vanishes because $U(1) = U(2) = 0$. Hence
\begin{align*}
    \sum_{n=2}^\infty (n-1)^\alpha R_p(U(n), \varphi_{Cop}(n)) = \sum_{m=2}^\infty m^\alpha R_p(u(m), \hat{\varphi}_{Cop}(m)).
\end{align*}
Combining the preceding identities yields
\begin{multline}\label{copson final}
    \sum_{n=1}^\infty n^\alpha |u(n) - u(n-1)|^p \ge \left( \frac{p-\alpha-1}{p} \right)^p \sum_{n=1}^\infty (n+1)^{\alpha-p} |u(n)|^p \\
    + \sum_{n=2}^\infty n^\alpha R_p(u(n), \hat{\varphi}_{Cop}(n)),
\end{multline}
where the constant $\left( \frac{p-\alpha-1}{p} \right)^p $ in (\ref{copson final}) is optimal, since the same constant is optimal in the corresponding inequality without the remainder term (see, \cite[Theorem 3]{stampach2026hardy}).
\end{proof}

\begin{proof}[\textbf{Proof of Corollary \ref{cor gupta}}]
Choosing 
\begin{align*}
v(n)=n^{\alpha} \quad \text{and} \quad \varphi(n)=\varphi_{\beta}(n)=n^{\beta} \quad \text{for} \quad \alpha \in \mathbb{R} \quad \text{and} \quad \beta> 0
\end{align*}
and substituting into the condition (\ref{cond}), we get
\begin{align*}
w(n) &= \frac{n^{\alpha}\bigl(n^{\beta} - (n-1)^{\beta}\bigr)^{p-1} - (n+1)^{\alpha}\bigl((n+1)^{\beta} - n^{\beta}\bigr)^{p-1}}{n^{\beta(p-1)}} \\
&= n^{\alpha} \left(\frac{n^{\beta} - (n-1)^{\beta}}{n^{\beta}}\right)^{p-1} - (n+1)^{\alpha} \left(\frac{(n+1)^{\beta} - n^{\beta}}{n^{\beta}}\right)^{p-1} \\
&= n^{\alpha} \left(1 - \left(1 - \frac{1}{n}\right)^{\beta}\right)^{p-1} - n^{\alpha}\left(\frac{n+1}{n}\right)^{\alpha} \left(\left(1 + \frac{1}{n}\right)^{\beta} - 1\right)^{p-1} \\
&= n^{\alpha} \left[\left(1 - \left(1 - \frac{1}{n}\right)^{\beta}\right)^{p-1} - \left(1 + \frac{1}{n}\right)^{\alpha} \left(\left(1 + \frac{1}{n}\right)^{\beta} - 1\right)^{p-1}\right]=w_{p,\alpha,\beta}(n)
\end{align*}
for $n\geq2$ and $w_{p,\alpha,\beta}(1):=1-2^{\alpha}\left(2^{\beta}-1\right)^{p-1}$. This matches the formula for $w_{p,\alpha,\beta}(n)$. Since the weight $w(n)$ was constructed so that the condition (\ref{cond}) holds as an equality for all $n \in \mathbb{N}$, the final statement of Theorem \ref{thm hardy identity} gives that the inequality (\ref{eq hardy identity}) becomes an identity. This proves \eqref{gen gupta main sec} for $\beta>0$.

When $p=2$, Remark \ref{rem lem} shows that the monotonicity assumption on $\phi_\beta$ is not required. Hence the same argument extends identity \eqref{gen gupta main sec} to every $\beta\in\mathbb R$, where, for $\beta\leq0$, the weight at $n=1$ is defined by
$$
w_{2,\alpha,\beta}(1)=1-2^\alpha(2^\beta-1),
$$
while the formula for $w_{2,\alpha,\beta}(n)$ is used for $n\geq2$. This completes the proof.
\end{proof}

{\bf Acknowledgments.} We are grateful to Shubham Gupta for his valuable comments which helped to improve the manuscript.

\bibliographystyle{alpha}
\bibliography{citation}

@article{kufner2006prehistory,
  title={The prehistory of the {H}ardy inequality},
  author={Kufner, A. and Maligranda, L. and Persson, L. E.},
  journal={Amer. Math. Monthly},
  volume={113},
  number={8},
  pages={715--732},
  year={2006},
  publisher={Taylor \& Francis}
}

@article{bennett1987elementary,
  title={Some elementary inequalities},
  author={Bennett, G.},
  journal={Quart. J. Math.},
  volume={38},
  number={4},
  pages={401--425},
  year={1987}
}

@article{giacomoni2022discrete,
  title={Discrete {P}icone inequalities and applications to non local and non homogeneous operators},
  author={Giacomoni, J. and Gouasmia, A. and Mokrane, A.},
  journal={Rev. R. Acad. Cienc. Exactas F{\'\i}s. Nat. Ser. A Mat. RACSAM},
  volume={116},
  pages={100},
  year={2022},
  publisher={Springer}
}

@article{brasco2014convexity,
  title={Convexity properties of {D}irichlet integrals and {P}icone-type inequalities},
  author={Brasco, L. and Franzina, G.},
  journal={Kodai Math. J.},
  volume={37},
  pages={769--799},
  year={2014}
}

@article{das2023improvements,
  title={On the improvements of {H}ardy and {C}opson inequalities},
  author={Das, B. and Manna, A.},
  journal={Rev. R. Acad. Cienc. Exactas F\'is. Nat. Ser. A Mat.},
  volume={117},
  number={2},
  pages={92},
  year={2023}
}

@article{amghibech2008discrete,
  title={On the discrete version of {P}icone's identity},
  author={Amghibech, S.},
  journal={Discrete Appl. Math.},
  volume={156},
  number={1},
  pages={1--10},
  year={2008},
  publisher={Elsevier}
}

@article{das2025improved,
  title={An improved {C}opson inequality},
  author={Das, B. and Manna, A.},
  journal={arXiv preprint arXiv:2508.00388},
  year={2025}
}

@article{das2025improvement,
  title={Improvement of the discrete weighted variant {H}ardy's inequality and criticality of the improved weight},
  author={Das, B. and Manna, A. and Paul, T.},
  journal={J. Geom. Anal.},
  volume={35},
  number={12},
  pages={376},
  year={2025}
}

@article{sano2017scale,
  title={Scale invariance structures of the critical and the subcritical {H}ardy inequalities and their improvements},
  author={Sano, M. and Takahashi, F.},
  journal={Calc. Var. Partial Differential Equations},
  volume={56},
  number={3},
  pages={69},
  year={2017}
}

@article{ioku2017hardy,
  title={{H}ardy type inequalities in $L^p$ with sharp remainders},
  author={Ioku, N. and Ishiwata, M. and Ozawa, T.},
  journal={J. Inequal. Appl.},
  volume={2017},
  number={1},
  pages={5},
  year={2017}
}

@article{sano2018scaling,
  title={Scaling invariant {H}ardy type inequalities with non-standard remainder terms},
  author={Sano, M.},
  journal={Math. Inequal. Appl.},
  volume={21},
  number={1},
  pages={77--90},
  year={2018}
}

@article{bez2018stability,
  title={Stability of trace theorems on the sphere},
  author={Bez, N. and Jeavons, C. and Ozawa, T. and Sugimoto, M.},
  journal={J. Geom. Anal.},
  volume={28},
  number={2},
  pages={1456--1476},
  year={2018}
}

@article{sano2018improvements,
  title={Some improvements for a class of the {C}affarelli--{K}ohn--{N}irenberg inequalities},
  author={Sano, M. and Takahashi, F.},
  journal={Differential Integral Equations},
  volume={31},
  number={1-2},
  pages={57--74},
  year={2018}
}

@article{yessirkegenov2026stability,
  title={Stability of the ${L}^{p}$-{P}oincar\'e inequality for the {L}ebesgue and {G}aussian probability measures with explicit geometric dependence and applications to spectral gaps},
  author={Yessirkegenov, N. and Zhangirbayev, A.},
  journal={arXiv preprint arXiv:2602.05968},
  year={2026}
}

@article{huang2025lp,
  title={On ${L}^p$-{H}ardy inequalities for magnetic $p$-{L}aplacians},
  author={Huang, Y. C. and Tong, X.},
  journal={arXiv preprint arXiv:2508.09483},
  year={2025}
}

@article{cazacu2024hardy,
  title={Hardy inequalities for magnetic $p$-{L}aplacians},
  author={Cazacu, C. and Krej{\v{c}}i{\v{r}}{\'\i}k, D. and Lam, N. and Laptev, A.},
  journal={Nonlinearity},
  volume={37},
  number={3},
  pages={035004},
  year={2024},
  publisher={IOP Publishing}
}

@article{CT24,
  title={Remainder {T}erms of ${L}^p$-{H}ardy {I}nequalities with {M}agnetic {F}ields: {T}he {C}ase $1 < p < 2$},
  author={Chen, X. P. and Tang, C. L.},
  journal={J. Geom. Anal.},
  volume={35},
  number={384},
  year={2025},
  publisher={Springer}
}

@article{ruzhansky2018note,
  title={A note on stability of {H}ardy inequalities},
  author={Ruzhansky, M. and Suragan, D.},
  journal={Ann. Funct. Anal.},
  volume={9},
  number={4},
  pages={451--462},
  year={2018}
}

@article{ruzhansky2019critical,
  title={Critical {H}ardy inequalities},
  author={Ruzhansky, M. and Suragan, D.},
  journal={Ann. Fenn. Math.},
  volume={44},
  number={2},
  pages={1159--1174},
  year={2019}
}

@article{banerjee2026sharp,
  title={Sharp quantitative forms of the {H}ardy inequality on {C}artan--{H}adamard manifolds via {S}obolev--{L}orentz embeddings},
  author={Banerjee, A. and Ganguly, D. and Roychowdhury, P.},
  journal={arXiv preprint arXiv:2601.13750},
  year={2026}
}

@article{cianchi2008hardy,
  title={{H}ardy inequalities with non-standard remainder terms},
  author={Cianchi, A. and Ferone, A.},
  journal={Ann. Inst. H. Poincar{\'e} Anal. Non Lin{\'e}aire},
  volume={25},
  number={5},
  pages={889--906},
  year={2008},
  publisher={Elsevier}
}

@article{figalli2022sharp,
  title={Sharp gradient stability for the {S}obolev inequality},
  author={Figalli, A. and Zhang, Y. R. Y.},
  journal={Duke Math. J.},
  volume={171},
  number={12},
  pages={2407--2459},
  year={2022}
}

@article{bonforte2025stability,
  title={Stability in {G}agliardo--{N}irenberg--{S}obolev inequalities: flows, regularity and the entropy method},
  author={Bonforte, M. and Dolbeault, J. and Nazaret, B. and Simonov, N.},
  journal={Mem. Amer. Math. Soc.},
  volume={308},
  number={1554},
  year={2025},
  publisher={American Mathematical Society}
}

@article{brasco2012sharp,
  title={Sharp stability of some spectral inequalities},
  author={Brasco, L. and Pratelli, A.},
  journal={Geom. Funct. Anal.},
  volume={22},
  number={1},
  pages={107--135},
  year={2012},
  publisher={Springer}
}

@article{dolbeault2025sharp,
  title={Sharp stability for {S}obolev and log-{S}obolev inequalities, with optimal dimensional dependence},
  author={Dolbeault, J. and Esteban, M. J. and Figalli, A. and Frank, R. L. and Loss, M.},
  journal={Camb. J. Math.},
  volume={13},
  number={2},
  pages={359--430},
  year={2025},
  publisher={International Press}
}

@article{cianchi2009sharp,
  title={The sharp {S}obolev inequality in quantitative form},
  author={Cianchi, A. and Fusco, N. and Maggi, F. and Pratelli, A.},
  journal={J. Eur. Math. Soc. (JEMS)},
  volume={11},
  number={5},
  pages={1105--1139},
  year={2009},
  publisher={European Mathematical Society}
}

@article{bianchi1991note,
  title={A note on the {S}obolev inequality},
  author={Bianchi, G. and Egnell, H.},
  journal={J. Funct. Anal.},
  volume={100},
  number={1},
  pages={18--24},
  year={1991},
  publisher={Elsevier}
}

@article{barki2024sharp,
  title={Sharp weighted discrete $p$-{H}ardy inequality and stability},
  author={Barki, A.},
  journal={arXiv preprint arXiv:2501.00299},
  year={2024}
}

@article{copson1928note,
  title={Note on series of positive terms},
  author={Copson, E. T.},
  journal={J. Lond. Math. Soc.},
  volume={3},
  number={1},
  pages={49--51},
  year={1928}
}

@article{frank2008non,
  title={Non-linear ground state representations and sharp {H}ardy inequalities},
  author={Frank, R. L. and Seiringer, R.},
  journal={J. Funct. Anal.},
  volume={255},
  number={12},
  pages={3407--3430},
  year={2008},
  publisher={Elsevier}
}

@misc{landau1921letter,
  title={Letter to {G}.{H}. {H}ardy},
  author={Landau, E.},
  year={June 21, 1921},
}

@article{stampach2026hardy,
  title={The $p$-{H}ardy--{R}ellich--{B}irman inequalities on the half-line},
  author={{\v{S}}tampach, F. and Waclawek, J.},
  journal={arXiv preprint arXiv:2603.08864},
  year={2026}
}

@article{gupta2024discrete,
  title={Discrete functional inequalities on lattice graphs},
  author={Gupta, S.},
  journal={arXiv preprint arXiv:2403.10270},
  year={2024}
}

@article{gupta2024onedimensional,
  title={One-dimensional discrete {H}ardy and {R}ellich inequalities on integers},
  author={Gupta, S.},
  journal={J. Fourier Anal. Appl.},
  volume={30},
  number={2},
  pages={15},
  year={2024}
}

@article{ciaurri2018hardy,
  title={{H}ardy's inequality for the fractional powers of a discrete {L}aplacian},
  author={Ciaurri, {\'O}. and Roncal, L.},
  journal={J. Anal.},
  volume={26},
  number={2},
  pages={211--225},
  year={2018}
}

@article{gerhat2025improved,
  title={An improved discrete {R}ellich inequality on the half-line},
  author={Gerhat, B. and Krej{\v{c}}i{\v{r}}{\'\i}k, D. and {\v{S}}tampach, F.},
  journal={Israel J. Math.},
  volume={268},
  number={1},
  pages={45--72},
  year={2025}
}

@article{gerhat2025criticality,
  title={Criticality transition for positive powers of the discrete {L}aplacian on the half line},
  author={Gerhat, B. and Krej{\v{c}}i{\v{r}}{\'\i}k, D. and {\v{S}}tampach, F.},
  journal={Rev. Mat. Iberoam.},
  volume={41},
  number={3},
  pages={1173--1200},
  year={2025}
}

@article{birman1961spectrum,
  title={On the spectrum of singular boundary-value problems},
  author={Birman, M. S.},
  journal={Mat. Sb. (N.S.)},
  volume={55(97)},
  pages={125--174},
  year={1961}
}

@article{stampach2024optimal,
  title={Optimal discrete {H}ardy-{R}ellich-{B}irman inequalities},
  author={{\v{S}}tampach, F. and Waclawek, J.},
  journal={arXiv preprint arXiv:2405.07742},
  year={2024}
}

@article{huang2024onedimensional,
  title={One-dimensional sharp discrete {H}ardy-{R}ellich inequalities},
  author={Huang, X. and Ye, D.},
  journal={J. Lond. Math. Soc.},
  volume={109},
  number={1},
  pages={e12838},
  year={2024}
}

@article{adimurthi2006optimal,
  title={Optimal {H}ardy-{R}ellich inequalities, maximum principles and related eigenvalue problems},
  author={Adimurthi and Grossi, M. and Santra, S.},
  journal={J. Funct. Anal.},
  volume={240},
  number={1},
  pages={36--83},
  year={2006},
  publisher={Elsevier}
}

@article{sekar2006role,
  title={Role of the fundamental solution in {H}ardy-{S}obolev-type inequalities},
  author={Adimurthi and Sekar, A.},
  journal={Proc. Roy. Soc. Edinburgh Sect. A},
  volume={136},
  number={6},
  pages={1111--1130},
  year={2006},
}

@article{ghoussoub2011bessel,
  title={Bessel pairs and optimal {H}ardy and {H}ardy--{R}ellich inequalities},
  author={Ghoussoub, N. and Moradifam, A.},
  journal={Math. Ann.},
  volume={349},
  pages={1--57},
  year={2011},
  publisher={Springer}
}

@article{cazacu2021method,
  title={The method of super-solutions in {H}ardy and {R}ellich type inequalities in the ${L}^2$ setting: an overview of well-known results and short proofs},
  author={Cazacu, C.},
  journal={Rev. Roumaine Math. Pures Appl.},
  volume={66},
  number={3-4},
  pages={617-638},
  year={2021}
}

@book{ghoussoub2013functional,
  title={Functional {I}nequalities: New {P}erspectives and New Applications},
  author={Ghoussoub, N. and Moradifam, A.},
  volume={187},
  year={2013},
  publisher={Amer. Math. Soc.}
}

@article{gupta2022discrete,
  title={Discrete weighted {H}ardy inequality in 1-{D}},
  author={Gupta, S.},
  journal={J. Math. Anal. Appl.},
  volume={514},
  number={2},
  pages={126345},
  year={2022}
}

@article{krejcirik2022sharp,
  title={A sharp form of the discrete {H}ardy inequality and the {K}eller-{P}inchover-{P}ogorzelski inequality},
  author={Krej{\v{c}}i{\v{r}}{\'\i}k, D. and {\v{S}}tampach, F.},
  journal={Amer. Math. Monthly},
  volume={129},
  number={3},
  pages={281--283},
  year={2022}
}

@article{fischer2023improved,
  title={An improved discrete $p$-{H}ardy inequality},
  author={Fischer, F. and Keller, M. and Pogorzelski, F.},
  journal={Integral Equations Operator Theory},
  volume={95},
  number={4},
  pages={24},
  year={2023}
}

@article{keller2018improved,
  title={An improved discrete {H}ardy inequality},
  author={Keller, M. and Pinchover, Y. and Pogorzelski, F.},
  journal={Amer. Math. Monthly},
  volume={125},
  number={4},
  pages={347--350},
  year={2018}
}

@article{rozenblum2009spectral,
  title={On the spectral estimates for the {S}chr{\"o}dinger operator on $\mathbb{Z}^d$, $d \ge 3$},
  author={Rozenblum, G. and Solomyak, M.},
  journal={arXiv preprint arXiv:0905.0270},
  year={2009}
}

@article{rozenblum2014spectral,
  title={On spectral estimates for the {S}chr{\"o}dinger operators in global dimension 2},
  author={Rozenblum, G. and Solomyak, M.},
  journal={St. Petersburg Math. J.},
  volume={25},
  number={3},
  pages={495--505},
  year={2014}
}

@article{kapitanski2016continuous,
  title={On continuous and discrete {H}ardy inequalities},
  author={Kapitanski, L. and Laptev, A.},
  journal={J. Spectr. Theory},
  volume={6},
  number={4},
  pages={837--858},
  year={2016}
}

@article{keller2018optimal,
  title={Optimal {H}ardy inequalities for {S}chr{\"o}dinger operators on graphs},
  author={Keller, M. and Pinchover, Y. and Pogorzelski, F.},
  journal={Comm. Math. Phys.},
  volume={358},
  number={2},
  pages={767--790},
  year={2018}
}

@article{gupta2023hardy,
  title={{H}ardy and {R}ellich inequality on lattices},
  author={Gupta, S.},
  journal={Calc. Var. Partial Differential Equations},
  volume={62},
  pages={81},
  number={3},
  year={2023}
}

@book{mazya2013sobolev,
  title={{S}obolev spaces},
  author={Maz'ya, V.},
  volume={342},
  year={2013},
  publisher={Springer}
}

@book{balinsky2015analysis,
  title={The analysis and geometry of {H}ardy's inequality},
  author={Balinsky, A. A. and Evans, W. D. and Lewis, R. T.},
  volume={1},
  year={2015},
  publisher={Springer}
}

@book{ruzhansky2019hardy,
  title={{H}ardy inequalities on homogeneous groups: {100} years of {H}ardy inequalities},
  author={Ruzhansky, M. and Suragan, D.},
  series={Progr. Math.},
  volume={327},
  year={2019},
  publisher={Birkh{\"a}user Cham}
}

@article{persson2024hardy,
  title={On {H}ardy-type inequalities as an intellectual adventure for 100 years},
  author={Persson, L. E. and Samko, N.},
  journal={J. Math. Sci.},
  volume={280},
  number={2},
  pages={180--197},
  year={2024},
  publisher={Springer}
}

@book{kufner2007hardy,
  title={The {H}ardy inequality: {A}bout its history and some related results},
  author={Kufner, A. and Maligranda, L. and Persson, L. E.},
  year={2007},
  publisher={Vydavatelsk{\'y} servis}
}

\end{document}